\documentclass[12pt]{article}
\usepackage{amsmath,amsxtra,latexsym,amsthm,amssymb,amscd,amsfonts}
\usepackage[portrait, top=3.5cm, bottom=3cm, left=3.5cm, right=2cm] {geometry}

\theoremstyle{plain}
\begin{document}
\pagestyle{myheadings}

\newtheorem{corollary}{\sc Corollary}[section]
\newtheorem{lemma}{\sc Lemma}[section]
\newtheorem{remark}{\it Remark}[section]
\newtheorem{proposition}{\sc Proposition}[section]
\newtheorem{definition}{\it Definition}[section]
\newtheorem{theorem}{\sc Theorem}[section]
\newtheorem{example}{\it Example}[section]

\renewcommand{\theequation}{\thesection.\arabic{equation}}
\normalsize

\setcounter{equation}{0}

\title{Efficient Solutions in Generalized Linear \\ Vector Optimization}

\author{Nguyen Ngoc Luan\footnote{Department of
		Mathematics and Informatics, Hanoi National University of Education, 136 Xuan Thuy, Hanoi, Vietnam; email: luannn@hnue.edu.vn.}}
\maketitle
\date{}

\medskip
\begin{quote}
\noindent {\bf Abstract.} This paper establishes several new facts on generalized polyhedral convex sets and shows how they can be used in vector optimization. Among other things, a scalarization formula for the efficient solution sets of generalized vector optimization problems is obtained. We also prove that the efficient solution set of a generalized linear vector optimization problem in a locally convex Hausdorff topological vector space is the union of finitely many generalized polyhedral convex sets and it is connected by line segments. 

\medskip


\end{quote}

\section{Introduction}
\markboth{\centerline{\it Introduction}}{\centerline{\it N.N.~Luan}} \setcounter{equation}{0}  

One calls a vector optimization problem {\it linear} if the objective function is linear and the constraint set is a polyhedral convex set. Due to the classical Arrow-Barankin-Blackwell theorem (the ABB theorem; see \cite{ABB_1953, Jahn, Luc}), for a finite dimensional linear vector optimization problem, the Pareto solution set and the weak Pareto solution set are connected by line segments and each of them is an union of finitely many faces of the constraint set. Extensions of the result for linear vector optimization problems in Banach spaces can be seen in \cite{Yang_Yen_2010, ZhengYang_2008}, where the focus point was piecewise linear vector optimization. In \cite{Zheng_1998}, it was shown that set of positive proper efficient points is dense in the set of efficient points with a pointed convex cone in a topological vector space.

\medskip
Scalarization methods, by which one replaces a vector optimization problem by a scalar optimization problem depending on a parameter, have attracted attentions of many researchers (see, e.g., Eichfelder in \cite{Eich08a}, Hoa, Phuong and Yen in \cite{Hoa_Phuong_Yen_2005}, Huong and Yen in \cite{Huong_Yen_2014}, Jahn in  \cite{Jahn, Jahn_1984}, Luc in \cite{Luc, Luc_1987, Luc2016}, Pascoletti and Serafini in \cite{Pascoletti_Serafini}, Yen and Phuong in \cite{Yen_Phuong_2000}, Zheng in \cite{Zheng_2000}). 

\medskip
Recently,  in locally convex Hausdorff topological vector spaces setting, using a representation for generalized polyhedral convex sets, Luan and Yen \cite{Luan_Yen_2015} have obtained solution existence theorems for generalized linear programming problems, a scalarization formula for the weakly efficient solution set of a generalized linear vector optimization problem, and proved that the latter is the union of finitely many generalized polyhedral convex sets. It is reasonable to look for similar results for the corresponding efficient solution set.  

\medskip
Our aim is to establish several new facts on generalized  polyhedral convex sets and shows how they can be used in vector optimization. Among other things, a scalarization formula for the efficient solution set of a generalized vector optimization problem is obtained. We also prove that the efficient solution set of a generalized linear vector optimization problem in a locally convex Hausdorff topological vector space is the union of finitely many generalized polyhedral convex sets and it is connected by line segments. The present paper can be considered as a continuation of \cite{Luan_Yen_2015}.

\medskip
The organization of our paper is as follows. Section 2 is devoted to an investigation on generalized polyhedral convex  sets. On that basis, Section 3 solves some questions about the efficient solution set of generalized linear vector optimization problems which arised after the paper by Luan and Yen \cite{Luan_Yen_2015}.

\section{Properties of Generalized Polyhedral Convex Sets}
\markboth{\centerline{\it Properties of Generalized Polyhedral Convex  Sets}}{\centerline{\it N.N.~Luan}} \setcounter{equation}{0}  

In this section, first we give a sufficient condition for the image of a generalized polyhedral convex set via a continuous linear map to be a generalized convex polyhedron. Second, we characterize the relative interior of a generalized polyhedral convex cone and of its dual cone. The obtained results will be used intensively in the sequel. 
	
\subsection{Images of  generalized  convex polyhedra}
Let $X$ be a {\it locally convex Hausdorff topological vector space} with the dual space denoted by $X^*$. For any $x^* \in X^*$ and $x \in X$, $\langle x^*, x \rangle$ indicates the value of $x^*$ at $x$. 

\begin{definition}{\rm (See \cite[p.~133]{Bonnans_Shapiro_2000})  A subset $C \subset X$ is said to be a \textit{generalized polyhedral convex set} (a \textit{generalized convex polyhedron} for short) if there exist $x^*_i \in X^*$, $\alpha_i \in \mathbb R$, $i=1,2,\dots,p$, and a closed affine subspace $L \subset X$, such that 
\begin{equation*}
			C=\big\{ x \in X \mid x \in L, \; \langle x^*_i, x \rangle \leq \alpha_i,\ \;  i=1,\dots,p\big\}.
\end{equation*} 
If $C$ admits the last representation for $L=X$ and for some $x^*_i \in X^*$, $\alpha_i \in \mathbb R$, $i=1,2,\dots,p$, then it is called a \textit{polyhedral convex set} (or a \textit{convex polyhedron}).}
\end{definition}
	
From the definition it follows that a generalized polyhedral convex set is a closed set. Note also that, in the finite dimensional space, $D$  is a generalized polyhedral convex set if and only if $D$ is a convex polyhedron.
	
The following representation theorem for generalized convex polyhedral in the spirit of \cite{Rock_book_1970} is crucial  for our subsequent proofs. 
	
\begin{theorem}{\rm  (\cite[Theorem~2.7]{Luan_Yen_2015})}\label{Luan_Yen_2015_main} A nonempty subset $D \subset X$ is a generalized convex polyhedron if and only if there exist $u_1, \dots, u_k \in X$, $v_1, \dots, v_{\ell} \in X$, and a closed linear subspace  $X_0 \subset X$ such that
\begin{equation}\label{rep_D}
		\begin{aligned}
		D=\Bigg\{ \sum\limits_{i=1}^k \lambda_i u_i + \sum\limits_{j=1}^\ell \mu_j v_j \mid & \lambda_i \geq 0, \ \forall i=1,\dots,k,    \\ 
		&\sum\limits_{i=1}^k \lambda_i=1,\ \, \mu_j \geq 0,\ \forall j=1,\dots,\ell   \Bigg\}+X_0.&& 
		\end{aligned}
\end{equation}
\end{theorem}
	
\medskip 
We are now in a position to extend Lemma 3.2 from the paper of Zheng and Yang \cite{ZhengYang_2008}, which was given in a normed spaces setting, to the case of convex polyhedra in locally convex Hausdorff topological vector spaces.
	
\begin{proposition}\label{image_of_D} If $T: X \rightarrow Y$ is a linear mapping between locally convex Hausdorff topological vector spaces with $Y$ being a space of finite dimension and if $D \subset X$ is a generalized polyhedral convex set, then $T(D)$ is a convex polyhedron of $Y$.		
\end{proposition}
	
\begin{proof} Suppose that $D$ is of the form \eqref{rep_D}. We have
	\begin{equation*}
	\begin{aligned}
	T(D)=\Bigg\{ \sum\limits_{i=1}^k \lambda_i (Tu_i) + \sum\limits_{j=1}^\ell \mu_j (Tv_j) \mid &\,  \lambda_i \geq 0, \ \forall i=1,\dots,k,    \\ 
	&\sum\limits_{i=1}^k \lambda_i=1,\ \, \mu_j \geq 0,\ \forall j=1,\dots,\ell   \Bigg\}+T(X_0).&& 
	\end{aligned}
	\end{equation*}
As $T(X_0)$ is a linear subspace of the finite dimensional space $Y$, $T(X_0)$ is a closed linear subspace. Hence, by Theorem \ref{Luan_Yen_2015_main}, $T(D)$ is a polyhedral convex set of $Y$.
\end{proof}
	
One may wonder: \textit{Whether the assumption on the finite dimensionality of $Y$ can be removed from Proposition \ref{image_of_D}, or not?} Let us solve this question by an example.
\begin{example}{\rm Let $X=C[0, 1]$ be the linear space of continuous real valued functions on the interval $[0, 1]$ with the norm defined by $||x||=\max\big\{|x(t)| \, \mid \, t \in [0,1] \big\}$.
Let $Y~=C_0[0,1]:=\big\{y \in C[0,1] \, \mid \, y(0)=0 \big\}$
and let $T: X \rightarrow Y$ be the bounded linear operator given by $(Tx)(t)=\int_{0}^{t}x(\tau) d \tau, $ where integral is Riemannian. Clearly, $X$ is a generalized polyhedral convex set in $X$ and
			\begin{equation*}
			T(X)=\Big\{ y \in C_0[0,1] \mid y \text{ is continuously differentiable on } (0,1) \Big\}.  
			\end{equation*}		
To show that $T(X)$ is dense in $Y$, we take any $y \in Y$. By the Stone-Weierstrass Theorem (see, e.g., \cite[Theorem~1.1, p.~52]{Lang_1993} and  \cite[Corollary~1.3, p.~54]{Lang_1993}), there exists a sequence of polynomial functions in one variable $\{p_k\}$ converging uniformly to~$y$ in~$Y$. Put $q_k(t)=p_k(t)-p_k(0)$ for all $t \in [0,1]$. It is easily seen that $\{q_k\}$ converges uniformly to $y$ in $Y$  and $\{q_k\} \subset T(X)$. As $T(X) \neq Y$, we see that $T(X)$ is a non-closed linear subspace set of~$Y$. Hence, $T(X)$ cannot be a generalized polyhedral convex set. 
}\end{example}
		
A careful analysis of Example 2.1 leads us to the following question: \textit{Whether the image of a generalized polyhedral convex set via a surjective linear operator from a Banach space to another Banach space is a generalized polyhedral convex set, or not?}
		
\begin{example}{\rm Let $X=C[0, 1] \times C[0,1]$ with the norm defined by 
				\begin{equation*}
				||(x, u)||=\max_{t \in [0, 1]} |x(t)|+\max_{t \in [0, 1]} |u(t)|,
				\end{equation*}
$Y=C[0,1]$ and a linear mapping $T: X \rightarrow Y$ be defined by
				\begin{equation*}
				T(x, u)(t)=\int_{0}^{t}x(\tau) d \tau + u(t), 
				\end{equation*}
where integral is Riemannian. Note that $D:=C[0,1] \times \{0\}$ is a generalized polyhedral convex set of $X$, but
				\begin{equation*}
				T(D)=\Big\{ y \in C_0[0,1] \mid y \text{ is continuously differentiable on } (0,1) \Big\}  
				\end{equation*}		
is not a generalized polyhedral convex set of $Y$. 
}\end{example}
			
\subsection{The relative interior of a polyhedral convex cone}
\begin{definition}{\rm (See \cite[p.~20]{Bonnans_Shapiro_2000})  For a convex subset $A$ of a locally convex Hausdorff topological vector space $X$, we say that a point $u$ belongs to the \textit{relative interior} of~$A$, denoted by ${\rm ri}A$, if there exists a neighborhood $U$ of $u$ in $X$ such that $U \cap {\rm cl(aff}A) \subset  A$, where ${\rm cl(aff}A)$ is the closure of the affine hull of $A$. }
\end{definition}
			
By \cite[Theorem 6.4]{Rock_book_1970} and \cite[Lemma 2.5]{Luan_Yen_2015}, if $X$ is a finite-dimensional Hausdorff topological vector space, and $A \subset X$ is a nonempty convex set, then $u \in {\rm ri}A$ if and only if, for every $x \in A$, there exists $\varepsilon >0$ such that $u-\varepsilon(x-u)$ belongs to $A$.  
			
One says that a subset $C$ of a locally convex Hausdorff topological vector space is a {\it cone} if $\lambda u \in C$ for all $u \in C$ and for every $\lambda >0$. Note that a cone may not contain~0.  
			
\begin{theorem}\label{reps_riC}  If $C \subset X$ is a generalized polyhedral convex cone in a locally convex Hausdorff topological vector space. If $C=\left\{ \sum\limits_{i=1}^p \lambda_i u_i \mid \lambda_i \geq 0,\ i=1,\dots,p  \right\}$, where $u_i \in X$ for $i=1,\dots,p$, then
				\begin{equation}\label{reps_riC_eq}
				{\rm ri}C=\left\{ \sum\limits_{i=1}^p \lambda_i u_i \mid \lambda_i > 0,\ i=1,\dots,p  \right\}.
				\end{equation}
\end{theorem}

\begin{proof} Let $X_0:={\rm{span}}\{u_1,\dots,u_p\}$ be the linear subspace generated by the vectors $u_1,\dots,u_p$. As $C$ is a convex cone of $X_0$ which is a space of finite dimension,  $u \in {\rm ri}C$ if and only if, for every $x \in C$, there exists $\varepsilon >0$ such that $u-\varepsilon(x-u) \in C$. Given any $u \in {\rm ri}C$, we will show that $u$ belongs to the right-hand-side of \eqref{reps_riC_eq}. Let $x=\sum\limits_{i=1}^p u_i \in C$ and let $\varepsilon >0$ be such that $v:=u-\varepsilon(x-u)$ belongs to $C$. Suppose that $v= \sum\limits_{i=1}^p \alpha_i u_i$, where $\alpha_i \geq 0$ for $i=1,\dots,p$. It is clear that
			\begin{equation*}
			u=\frac{1}{1+\varepsilon}v+\frac{\varepsilon}{1+\varepsilon}x\\
			=\sum\limits_{i=1}^{p}\left(\frac{\alpha_i}{1+\varepsilon}+\frac{\varepsilon}{1+\varepsilon} \right)u_i= \sum\limits_{i=1}^{p} \lambda_i u_i,
			\end{equation*}
where $\lambda_i:=\frac{\alpha_i}{1+\varepsilon}+\frac{\varepsilon}{1+\varepsilon} >0$, $i=1,\dots,p$. This establishes the inclusion ``$\subset$'' in \eqref{reps_riC_eq}. 
			
Now, let $u$ be an arbitrary element from the set on the right-hand-side of \eqref{reps_riC_eq}. Suppose that $u=\sum\limits_{i=1}^{p} \lambda_i u_i$, where $\lambda_i>0$ for all  $i=1,\dots,p$. For any $x \in C$, one can choose $\alpha_1 \geq 0, \dots, \alpha_p \geq 0$, satisfying $x=\sum\limits_{i=1}^{p} \alpha_i u_i$. Put
			\begin{equation*}
			v_{\varepsilon}=u-\varepsilon(x-u)=\sum\limits_{i=1}^{p} \big(\lambda_i -\varepsilon(\alpha_i - \lambda_i) \big)u_i, 
			\end{equation*}
where $\varepsilon >0$. As $\lambda_i >0$ for all  $i=1,\dots,p$, we can find an $\varepsilon >0$ satisfying $\lambda_i -\varepsilon(\alpha_i - \lambda_i) \geq 0$ for all $i=1,\dots,p$. Hence, for this $\varepsilon$, we have $v_{\varepsilon} \in C$. The inclusion ``$\supset$'' in \eqref{reps_riC_eq} has been proved.
\end{proof}	
			
\medskip
Let $Y$ be a locally convex Hausdorff topological vector space. Suppose that $K \subset Y$ is a polyhedral convex cone defined by
			\begin{equation}\label{reps_K}
			K=\Big\{ y \in Y \mid \langle y^*_j, y \rangle \leq 0,\ j=1,\dots,q  \Big\}, 
			\end{equation}
where $y^*_j \in Y^*\setminus\{0\}$ for all $j=1,\dots,q$. Define $\ell(K)=K \cap (-K)$. It is clear that
			\begin{equation*}
			\ell(K)=\big\{ y \in Y \mid \langle y^*_j, y \rangle = 0,\ j=1,\dots,q  \big\}.
			\end{equation*}
			
The first assertion of the following proposition describes the interior of a polyhedral convex cone.
					
\begin{proposition}\label{reps_intK}  Let $K \subset Y$ be a polyhedral convex cone of the form \eqref{reps_K}. The following are valid:
			\begin{description}
					\item{\rm (a)} The interior of $K$ has the represention 
							\begin{equation}\label{reps_intK_eq}
									{\rm int}K=\big\{ y \in Y \mid \langle y^*_j, y \rangle < 0,\ j=1,\dots,q  \big\}.
							\end{equation}
					\item{\rm (b)} The set $K \setminus{\ell(K)}$ is a convex cone and
							\begin{equation}\label{reps_1b}
					K \setminus{\ell(K)}=\Big\{ y \in K \mid \text{there exists $j \in \{1,\dots,q\}$ such that } \langle y^*_j, y \rangle< 0  \Big\}.
							\end{equation}	
			\end{description}
\end{proposition}

\begin{proof} (a)  As $ \big\{ y \in Y \mid \langle y^*_j, y \rangle < 0,\ j=1,\dots,q  \big\} $ is an open subset of $K$, we have the inclusion ``$\supset$'' in \eqref{reps_intK_eq}. Now, to obtain the reverse inclusion, arguing by contradiction, we suppose that there exists $\bar{y} \in {\rm int}K$ for which there is $j_1 \in \{1,\dots,q\}$ such that $  \langle y^*_{j_1},  \bar{y} \rangle=0$. Since $\bar{y} \in {\rm int}K$, one can find a balanced neighborhood $V \subset Y$ of 0 satisfying $\bar{y}+V \subset K$. Then we have
			\begin{equation*}
			0 \geq \langle y^*_{j_1},  \bar{y} +v \rangle= \langle y^*_{j_1}, v \rangle
			\end{equation*}
for all $v \in V$.  It follows that $\langle y^*_{j_1}, v \rangle = 0$ for all $v \in V$. As $y^*_{j_1} \neq 0$,  there exists $y \in Y$ with $\langle y^*_{j_1}, y \rangle \not=0$. Since $ty \in V$ for sufficiently small $t>0$, we get $\langle y^*_{j_1}, y \rangle=0$, a~contradiction.  Thus, we have proved the inclusion ``$\subset$'' in \eqref{reps_intK_eq}.
			
\medskip
(b) Clearly, $y \in K \setminus{\ell(K)}$ if and only if   $\langle y^*_j, y \rangle \leq~0$ for all $ j=1,\dots,q,$ and there exists $j \in \{1,\dots,q\}$ such that $\langle y^*_j, y \rangle< 0$.  \eqref{reps_1b} holds true. The fact that $K \setminus{\ell(K)}$ is a cone is obvious. Hence to show that $K \setminus{\ell(K)}$ is convex, we take any $u, v \in K \setminus{\ell(K)}$ and $\lambda \in (0,1)$. By the convexity of $K$, $\lambda u + (1-\lambda) v \in K$. As $u \in  K \setminus{\ell(K)}$, one can find an index $j_0 \in \{1,\dots,q\}$ such that $ \langle y^*_{j_0}, u \rangle< 0 $. Since
			\begin{equation*}
			\langle y^*_{j_0}, \lambda u + (1-\lambda) v \rangle = \lambda \langle y^*_{j_0}, u \rangle + (1-\lambda) \langle y^*_{j_0}, v \rangle \leq \lambda \langle y^*_{j_0}, u \rangle  < 0,
			\end{equation*}
we have $\lambda u + (1-\lambda) v \in K \setminus{\ell(K)}$. 
\end{proof}
			
\begin{remark}\label{col_reps_intK}{\rm From Proposition \ref{reps_intK} it follows that ${\rm int}K \subset K \setminus{\ell(K)}$. The last inclusion can be strict. To see this, choose $Y=\mathbb{R}^n$ and $K=\mathbb{R}^n_+$.}
\end{remark}
			
To proceed furthermore, we put $Y_0=\big\{ y \in Y \, \mid \,  \langle y^*_j, y \rangle =0,\ j=1,\dots,q \big\}$ and note that $\ell(K)=Y_0$. Because $Y_0$ is a closed linear subspace of finite codimension of~$Y$, there exists a finite-dimensional linear subspace $Y_1$ of $Y$, such that $Y=Y_0 + Y_1$ and $Y_0 \cap Y_1=\{0\} $.  By \cite[Theorem 1.21(b)]{Rudin_1991}, $Y_1$ is closed. Clearly, 
	\begin{equation*}
	K_1:=\big\{ y \in Y_1 \mid \langle y^*_j, y \rangle \leq 0, \; j=1,\dots,q\big\}
	\end{equation*}
is a pointed polyhedral convex cone in $Y_1$ and $K=Y_0+K_1$. 

\medskip
For the case $Y=\mathbb{R}^n$, a result similar to the following one was given in \cite[Lemma~2.6,~p. 89]{Luc}.
\begin{lemma}\label{decomp_K_diff_lK} It holds that 
				\begin{equation}\label{decomp_for_K_diff_lK_eq}
				K \setminus \ell(K)=Y_0+K_1 \setminus \{0\}.
				\end{equation}	
\end{lemma}
		
\begin{proof} For each $v \in K \setminus \ell(K)$, there exist $v_0 \in Y_0$ and $v_1 \in K_1$ satisfying $v=~v_0+~v_1$. By Proposition \ref{reps_intK}, one can find $j_0 \in \{1,\dots,q\}$ such that $\langle y^*_{j_0}, v \rangle < 0$. Since $\langle y^*_{j_0}, v_1 \rangle =   \langle y^*_{j_0}, v \rangle - \langle y^*_{j_0}, v_0 \rangle <0$,  so $v_1$ is non zero. Hence $v=v_0+v_1 \in Y_0 +  K_1 \setminus \{0\}$. We have shown that $K \setminus \ell(K) \subset Y_0+K_1 \setminus \{0\}$.
			
To obtain \eqref{decomp_for_K_diff_lK_eq}, take any $v=v_0+v_1$ with $v_0 \in Y_0$ and $v_1 \in K_1 \setminus \{0\}$. Then $v \in Y_0+K_1=K$.  As $Y_0 \cap \big(Y_1 \setminus \{0\} \big)  =\emptyset$, we must have $v_1 \not \in Y_0$. Choose $j_1~\in~\{1,\dots,q\}$ such that $\langle y^*_{j_1}, v_1 \rangle <0$.  Since $\langle y^*_{j_1}, v \rangle =   \langle y^*_{j_1}, v_0 \rangle + \langle y^*_{j_1}, v_1 \rangle <0$,  we see that $v \in K \setminus \ell(K)$.
\end{proof}

\medskip
By \cite[Proposition 2.42] {Bonnans_Shapiro_2000}, we can represent the positive dual cone 
			\begin{equation*}
			K^*:= \big\{ y^* \in Y^* \mid \langle y^*, y \rangle \geq 0 \ \; \forall y \in K   \big\}
			\end{equation*}
of $K$ as
			\begin{equation}\label{represent_Kstar}
			K^*=\left\{\sum\limits_{j=1}^q \lambda_j (-y^*_j) \mid \lambda_j \geq 0, j=1,\dots,q  \right\}.
			\end{equation}
			
\begin{lemma}\label{Kstar_property}
				If $y^* \in K^*$ then $\langle y^*, y \rangle =0$ for all $y \in Y_0$.
\end{lemma}

\begin{proof} If $y^* \in K^*$ then, for any $y \in Y_0$, one has $\langle y^*, y \rangle  \geq 0$ and $\langle y^*, -y \rangle  \geq 0$; hence $\langle y^*, y \rangle =0$. 
\end{proof}
			
\medskip
Now we are in position to describe the relative interior of the dual cone $K^*$ by using the set $K\setminus{\ell(K)}$, which can be computed by \eqref{reps_1b}.  
\begin{theorem}\label{decomp_riKstar}
				If $K$ is not a linear subspace of $Y$,	then a vector $y^* \in Y^*$ belongs to ${\rm ri}K^*$ if and only if $\langle y^*, y \rangle >0$ for all $y \in K\setminus{\ell(K)}$. 
\end{theorem}
			
\begin{proof}  {\it Necessity:} Suppose that $y^* \in {\rm ri}K^*$. By Theorem \ref{reps_riC} and formula \eqref{represent_Kstar}, there exist $\lambda_j >0$ for $j=1,\dots,q$, such that  $y^*=\sum\limits_{j=1}^q \lambda_j (-y^*_j)$. For any $y \in K\setminus{\ell(K)}$, by Proposition~\ref{reps_intK} one can find $j_0 \in \{1,\dots,q\}$ satisfying $\langle y^*_{j_0}, y \rangle< 0$. Then we have
			\begin{equation*}
			\begin{aligned}
			\langle y^*, y \rangle &= (-\lambda_{j_0}) \langle y^*_{j_0}, y \rangle + \sum\limits_{j=1, j\neq j_0}^q (-\lambda_j) \langle y^*_{j}, y \rangle\\
			&\geq (-\lambda_{j_0}) \langle y^*_{j_0}, y \rangle >0,
			\end{aligned}
			\end{equation*}
as derised.
			
{\it Sufficiency:} Suppose that $y^* \in Y^*$ and $\langle y^*, y \rangle >0$ for all $y \in K\setminus{\ell(K)}$. To show that $y^* \in K^*$, we assume the contrary: There exists $\bar{y} \in K$ with $\langle y^*, \bar{y} \rangle  <0$. Since $\langle y^*, y \rangle >0$ for all $y \in K\setminus{\ell(K)}$, this inequality forces $\bar{y} \in \ell(K) = Y_0$. Given any $y_1 \in K_1 \setminus\{0\}$,  by  \eqref{decomp_for_K_diff_lK_eq}, we have $t\bar{y}+y_1 \in K \setminus{\ell(K)}$ for every $t \in \mathbb{R}$.  As $\langle y^*, \bar{y} \rangle <0$, we can find $t >0$ such that 
			\begin{equation*}
			\langle y^*, t\bar{y}+y_1 \rangle =t\langle y^*, \bar{y}\rangle+\langle y^*, y_1 \rangle <0.
			\end{equation*}
This contradicts the hypothesis that $\langle y^*, y \rangle >0$ for all $y \in K\setminus{\ell(K)}$. Thus $y^* \in K^*$.
			
Since $K$ is not a linear supspace of $Y$, $K_1 \neq \{0\}$. By \cite[Theorem~19.1]{Rock_book_1970}, one can find $v_i \in Y_1 \setminus\{0\}, \, i=1,\dots,\ell,$ such that  
			\begin{equation*}
			K_1=\left\lbrace\sum\limits_{i=1}^{\ell} \mu_i v_i   \; \mid \; \mu_i \geq 0, \; \ i=1,\dots, \ell\right\rbrace. 
			\end{equation*}
Since $v_i \in K \setminus{\ell(K)}$ for $i=1,\dots,\ell$, by \eqref{decomp_for_K_diff_lK_eq}, it follows that
			\begin{equation}\label{y_star_v_i}
			\langle y^*, v_i  \rangle>0, \; i=1,\dots,\ell.
			\end{equation}
Take any $\widetilde{y}^* \in K^*$ and put $v^*_{\varepsilon}=y^* - \varepsilon (\widetilde{y}^*-y^*)$ with $\varepsilon >0$. By \eqref{y_star_v_i}, there exists $\varepsilon >0$ such that 
			\begin{equation}\label{v_star_v_i}
			\langle v^*_{\varepsilon}, v_i  \rangle>0,\; \  i=1,\dots,\ell.
			\end{equation}
As $K=Y_0+K_1$,  for every $y \in K$ one can find $y_0 \in Y_0$ and $\mu_i \geq 0$, $i=1,\dots,\ell,$ such that $y=y_0+\sum\limits_{i=1}^{\ell}\mu_i v_i$. Because $y^*, \widetilde{y}^* \in K^*$, by Lemma \ref{Kstar_property} one has $\langle y^*, y_0 \rangle =0$ and $\langle \widetilde{y}^*, y_0 \rangle =0$. Hence, from \eqref{y_star_v_i} it follows that
			\begin{equation*}
			\begin{aligned}
			\langle v_{\varepsilon}^*, y \rangle &= \left\langle v_{\varepsilon}^* , y_0+\sum\limits_{i=1}^{\ell}\mu_i v_i \right\rangle\\
			&=\sum\limits_{i=1}^{\ell}\mu_i  \langle  v^*_{\varepsilon}, v_i \rangle \geq 0.
			\end{aligned}
			\end{equation*} 
So we have $ \langle v_{\varepsilon}^*, y \rangle \geq 0$ for every $y \in K$. This means that $v_{\varepsilon}^* \in K^*$. We have thus  proved that, for  any $\widetilde{y}^* \in K^*$, there exists $\varepsilon >0$ such that $y^*-\varepsilon(\widetilde{y}^*-y^*) \in K^*$. Since $K^*$ is a convex cone in the finite dimensional space ${\rm span}\{y^*_1,\dots,y_q^*\}$, by \cite[Theorem~6.4]{Rock_book_1970}, we can infer that $y^* \in {\rm ri}K^*$.   
\end{proof}
			
\begin{remark}{\rm
If $K=Y_0$, then $K^*=Y_0^{\perp}$ where  
					\begin{equation*}
					Y_0^{\perp}:=\{y^* \in Y^* \mid \langle y^*, y \rangle = 0, \ \forall y \in Y_0\}
					\end{equation*}
is the annihilator of $Y_0$. So we have ${\rm ri}K^*=K^*$. 
}\end{remark}
				
\section{Efficient Solutions}
\markboth{\centerline{\it Efficient Solutions}}{\centerline{\it N.N.~Luan}} \setcounter{equation}{0}  				
				
Following \cite{Luan_Yen_2015}, we consider a {\it generalized linear vector optimization problem}   
\begin{equation*}
				{\rm (VLP)} \qquad \quad {{\rm min}_K} \big\{Mx \mid x \in D\big\},
\end{equation*}
where $M: X \rightarrow Y$ is a continuous linear mapping between locally convex Hausdorff topological vector spaces, $D \subset X$ a generalized polyhedron, $K \subset Y$ a polyhedral convex cone of the form \eqref{reps_K}. 
				
\medskip
A vector $u \in D$ is said to be an \textit{efficient solution} (resp.,  a \textit{weakly efficient solution}) of {\rm (VLP)}  if there does not exist any $x \in D$ such that $Mu - Mx \in K\setminus{\ell(K)}$  (resp.,  $Mu - Mx \in {\rm int}K$). The set of all the efficient solutions (resp.,  \textit{weakly efficient solution}) is denoted by $E$ (resp.,  $E^w$). 
				
Clearly, when $K$ is a pointed cone, i.e., $\ell(K)=\left\{ 0 \right\}$, then $u \in E$ if and only if there does not exist any $x \in D$ with $Mu - Mx \in K\setminus\{0\}$.  
				
\begin{remark}{\rm As ${\rm int} K \subset K\setminus{\ell(K)}$ by Remark \ref{col_reps_intK}, we have $E \subset E^w$.}
\end{remark}
				
Now, by a standard scalarization scheme in vector optimization, we consider the scalar problems
				\begin{equation*}
				{\rm (LP)}_{y^*} \qquad \min \Big\{\langle y^*, Mx \rangle \mid x \in D\Big\} \quad (y^* \in Y^*).
				\end{equation*}
				
\medskip
Let $\pi_0: Y \rightarrow Y/Y_0$, $y \mapsto y+Y_0$ for all $y \in Y$, be the canonical projection from $Y$ on the quotient space $Y/Y_0$. It is clear that the operator $\Phi_0: Y/Y_0 \rightarrow Y_1$, $[y_1]:=y_1 +Y_0 \mapsto y_1$ for all $y_1 \in Y_1$, is a linear bijective mapping. By \cite[Theorem~1.41({\it a})]{Rudin_1991}, $\pi_0$ is a linear continuous mapping. Moreover, $\Phi_0$ is a homeomorphism by \cite[Lemma~2.5]{Luan_Yen_2015}. Hence, the operator $\pi:=\Phi_0\circ \pi_0: Y \rightarrow Y_1$ is linear and continuous. Therefore, by Proposition \ref{image_of_D}, $D_1:=\left(\pi \circ M\right) (D)$ is a convex polyhedron in $Y_1$.
				
We now show how of checking the inclusion $u \in E$, for every $u \in D$, verification a relation in the finite dimensional space $Y_1$.  
				
\begin{proposition}\label{relative_eff}
For any $u \in D$, one has $u \in E$ if and only if 
					\begin{equation}\label{relative_eff_eq1}
					\big(\left(\pi \circ M\right) u- D_1\big)  \cap \big( K_1\setminus\{0\}\big) =~\emptyset.
					\end{equation}	
\end{proposition}
				
\begin{proof} {\it Necessity:} Suppose the contrary that there is some $u \in E$ with 
				\begin{equation*}
				\left( u_1 - D_1\right)  \cap \big( K_1\setminus\{0\}\big) \neq~\emptyset,
				\end{equation*}	
where $u_1:=\left(\pi \circ M\right) u=\pi(Mu)$. Setting $u_0=Mu-u_1$, we have $Mu=u_0+u_1$ with $u_1=\pi(Mu) \in Y_1$. So $y_0 \in Y_0$. Select an element $v_1 \in K_1 \setminus\{0\}$ such that $v_1 \in u_1 - D_1$.  As $u_1-v_1 \in D_1=\left( \pi \circ M\right) (D)$, there exist $y \in D$ and $y_0 \in Y_0$ satisfying $My=y_0+(u_1-v_1)$. Since
				\begin{equation*}
				Mu-My = (u_0+u_1)-\big(y_0+(u_1-v_1)\big)=(u_0-y_0)+v_1 \in Y_0 + K_1 \setminus\{0\},
				\end{equation*}
by Lemma \ref{decomp_K_diff_lK}, we have $Mu-My \in  K \setminus{\ell(K)} $. This contradicts the assumption $u \in  E$. We have thus proved that if $u \in E$ then \eqref{relative_eff_eq1} holds. 
				
\medskip
{\it Sufficiency:} Ab absurdo, suppose that there exists $u \in D$ satisfying \eqref{relative_eff_eq1},  but $u \not \in E$. As $\big(Mu-M(D)\big) \cap  \big( K \setminus{\ell(K)} \big)  \neq \emptyset$, one can find $y \in D$ and $v \in K \setminus{\ell(K)}$ satisfying $Mu-My=v$. Invoking Lemma \ref{decomp_K_diff_lK}, we can assert that $v \in Y_0+K_1 \setminus\{0\}$, i.e., $v=v_0+v_1$ for some $v_0 \in Y_0$ and $v_1 \in K_1 \setminus\{0\}$. Then, from the equality $Mu-My=v_0+v_1$ we get
				\begin{equation*}
				\begin{aligned}
				\pi(Mu)-\pi(My)=\pi(Mu-My)&=\pi(v_0+v_1)\\	
				&=\pi(v_0)+\pi(v_1)=v_1.
				\end{aligned}
				\end{equation*}
It follows that $v_1 \in \Big( \pi(Mu)  -\pi(M(D))  \Big) \cap \Big( K_1 \setminus\{0\}\Big) $.  This is incompatible with \eqref{relative_eff_eq1}. The proof is complete. 
\end{proof}
				
\medskip
To make this exposition comprehensive, we how have a new look on a technical lemma of \cite{Luc} by giving another proof for it. 
				
\begin{lemma}\label{sclar_finite_dimen} {\rm (\cite[Lemma 2.6, p.~89]{Luc})} Suppose that $Z$ is a finite dimensional locally convex Hausdorff topological vector space. Let $A$ be a convex polyhedron containing~0 and $K \subset Z$ be a pointed polyhedral convex cone. If $A \cap\big(  K\setminus\{0\}\big)  =\emptyset$, then there exists $z^* \in Z^*$ such that 
					\begin{equation}\label{sclar_finite_dimen_eq1}
					\langle z^*, z  \rangle \leq 0 < \langle z^*,v \rangle
					\end{equation}    	
for all $z \in A$ and for any $v \in K\setminus\{0\}$.
\end{lemma}	
				
\begin{proof}  Suppose that $K=\big\{ z \in Z \mid \langle z^*_j,z \rangle \leq 0,  \ j=1,\dots,q   \big\}$, where $z^*_j \in Z^*$ for $j=1,\dots,q.$ It is clear that 
				\begin{equation*}
				\begin{aligned}
				B:=\Bigg\{  z \in K \mid \sum_{j=1}^{q}\langle z^*_j,z \rangle =-1 \Bigg\}&& 
				\end{aligned}
				\end{equation*} 
is a compact convex polyhedron and $K \setminus\{0\} = \bigcup\limits_{t >0}(tB)$. According to \cite[Theorem~19.1]{Rock_book_1970}, there exist $u_i \in A$, $i=1,\dots,k$, $d_j \in Z$, $j=1,\dots,r,$ such that 
				\begin{equation*}
				\begin{aligned}
				A=\Bigg\{ \sum\limits_{i=1}^k \lambda_i u_i + \sum\limits_{j=1}^r \gamma_j d_j \mid &\,  \lambda_i \geq 0, \ \forall i=1,\dots,k,    \\ 
				&\sum\limits_{i=1}^k \lambda_i=1,\ \, \gamma_j \geq 0,\ \forall j=1,\dots,r   \Bigg\}.&& 
				\end{aligned}
				\end{equation*}
Since $0 \in A$, we must have $d_j \in A$ for $j=1,\dots,r$. Consider the cone
				\begin{equation*}
				\begin{aligned}
				C:=\Bigg\{ \sum\limits_{i=1}^k \lambda_i u_i + \sum\limits_{j=1}^r \gamma_j d_j \mid &\,  \lambda_i \geq 0, \ \forall i=1,\dots,k,    \gamma_j \geq 0,\ \forall j=1,\dots,r   \Bigg\}. 
				\end{aligned}
				\end{equation*}
Clearly, $C$ is a  pointed polyhedral convex cone and $A \subset C$. Note that, for any $z \in C$, there exist $u \in A$ and $\delta >0$ such that $z=\delta u$.  (Indeed, given any $z \in C$, one can find $\lambda_i \geq 0$, $i=1,\dots,k,$ and $\gamma_j \geq 0$, $j=1,\dots,r,$ such that $z= \sum\limits_{i=1}^k \lambda_i u_i + \sum\limits_{j=1}^r \gamma_j d_j$. If $\lambda_i=0$ for all $i=1,\dots,k$, then $z= \sum\limits_{j=1}^r \gamma_j d_j \in A$; hence we can choose $u=z$ and $\delta=1$. If there exists $\lambda_i >0$ then we choose $\delta = \sum\limits_{i=1}^k \lambda_i >0$ and $u=\frac{1}{\delta}z \in A$.) Since $A \cap (K\setminus\{0\}) =\emptyset$, by our assumptions,  we have $B \cap C =\emptyset$. Since $B, C$ are closed convex subsets of $Z$ and $B$ is compact,  by the strongly separation theorem  (see, e.g., \cite[Theorem 2.~14]{Bonnans_Shapiro_2000}), there exists $z^* \in Z^*$ such that 
				\begin{equation}\label{sep_C_and_B}
				\sup\limits_{z \in C} \langle z^*, z \rangle < \inf\limits_{b \in B} \langle z^*, b \rangle. 
				\end{equation}       
Since $0 \in C$, $0 \leq \sup\limits_{z \in C} \langle z^*, z \rangle$. Hence, from \eqref{sep_C_and_B} it follows that $0 < \inf\limits_{b \in B} \langle z^*, b \rangle$. Therefore, $ \langle z^*, v \rangle > 0$ for all $v \in K\setminus\{0\}$. For every $z \in C$, the inequality $\sup\limits_{t >0} \langle z^*, tz \rangle <~\inf\limits_{b \in B} \langle z^*, b \rangle$ forces $\langle z^*, z \rangle \leq 0$. Hence, \eqref{sclar_finite_dimen_eq1} is valid.
\end{proof}
				
\begin{theorem}\label{scalar_theorem}
If $K$ is not a linear subspace of $Y$, then $u \in Y$ is an efficient solution of {\rm (VLP)} if and only if there exists $y^* \in {\rm ri}K^*$ satisfying 	$u \in {\rm argmin} \left( {\rm (LP)}_{y^*}\right)$. In other words,
					\begin{equation}\label{repr_efficient_solution_set_1}
					E=\bigcup\limits_{y^* \in {\rm ri}K^*} {\rm argmin} \left( {\rm (LP)}_{y^*}\right).
					\end{equation}
\end{theorem}
				
\begin{proof} If $u \in E$, then \eqref{relative_eff_eq1} holds by Proposition \ref{relative_eff}. According to Proposition \ref{image_of_D} and \cite[Corollary 19.3.2]{Rock_book_1970}, $\pi(Mu) - D_1$ is a convex polyhedron in $Y_1$. Using Lemma \ref{sclar_finite_dimen} for the convex polyhedron $\pi(Mu) - D_1$ corresponding the pointed polyhedral convex cone $K_1$ in $Y_1$, one can find $v^* \in Y_1^*$ such that 
				\begin{equation}\label{scalar_theorem_eq1}
				\langle v^*, w  \rangle \leq 0 < \langle v^*,v \rangle, \quad \forall w \in \pi(Mu) - D_1, \forall v \in K_1\setminus\{0\}.
				\end{equation}    	
Setting $y^*=v^* \circ \pi$ and note that $y^* \in Y^*$. For any $x \in D$, since  $\pi(Mu) - \pi(Mx)  \in \pi(Mu) - D_1$, we have
				\begin{equation*}
				\langle y^*, Mu-Mx  \rangle = \langle v^*, \pi(Mu)-\pi(Mx) \rangle  \leq 0.
				\end{equation*}
Hence, we obtain $\langle y^*, Mu \rangle \leq \langle y^*, Mx \rangle$ for all $x \in D$; so $u \in {\rm argmin} \left( {\rm (LP)}_{y^*}\right)$. Let us show that $y^* \in {\rm ri} K^*$. Given any $y \in K\setminus {\ell(K)}$, by Lemma \ref{decomp_K_diff_lK} one can find $y_0 \in Y_0$ and $y_1 \in  K_1\setminus\{0\}$ such that $y=y_0+y_1$. Then
				\begin{equation*}
				\langle y^*, y  \rangle = \langle v^*, \pi(y)  \rangle =\langle v^*, y_1  \rangle >0
				\end{equation*}  
by \eqref{scalar_theorem_eq1}. By Theorem \ref{decomp_riKstar}, $y^* \in {\rm ri}K^*$.  The inclusion $E \subset \bigcup\limits_{y^* \in {\rm ri}K^*} {\rm argmin} \left( {\rm (LP)}_{y^*}\right)$ has been established. 
				
Now, to obtain the reverse inclusion, suppose on contrary that there exists $u \in {\rm argmin} \left( {\rm (LP)}_{y^*}\right)$, with $y^* \in {\rm ri}K^*$, but $ u  \not \in E$. Select an $x \in D$ such that $Mu - Mx \in K\setminus{\ell(K)}$. Then, $\langle y^*, Mu-Mx \rangle >0$ by Theorem \ref{decomp_riKstar}. This contradicts the condition $u \in {\rm argmin} \left( {\rm (LP)}_{y^*}\right)$. The proof of \eqref{repr_efficient_solution_set_1} is thus complete. 		
\end{proof}	
				
\medskip
The scalarization formula  \eqref{repr_efficient_solution_set_1} allows us to obtain the following result on the structure of the efficient solution set of {\rm (VLP)}. 
				
\begin{theorem}\label{structure_efficient_solution_set} The efficient solution set $E$ of {\rm (VLP)} is the union of finitely many generalized polyhedral convex sets.
\end{theorem}
				
\begin{proof} The conclusion follows from \eqref{repr_efficient_solution_set_1} and an argument similar to that of the proof of \cite[Theorem 4.5]{Luan_Yen_2015}. 
\end{proof}
				
\medskip
If the spaces in question are finite dimensional, then the result in   Theorem \ref{structure_efficient_solution_set} expresses one conclusion of the Arrow-Barankin-Blackwell Theorem. The second assertion the latter is that $E$ is connected by line segments. A natural question arises: \textit{Whether the efficient solution set $E$ of {\rm (VLP)} is connected by line segments, or not?}
				
\medskip	
According to \cite{Luc}, the connected by line segments of the efficient solution set $E$ in finite dimensional setting can be proved by a scheme the suggested by Podinovski and Nogin \cite{Podinovski_Nogin}. We now show that an adaption of the scheme on show work for the locally convex Hausdorff topological vector spaces setting which we are interested in. 

\begin{theorem}\label{eff_sol_arcwise}
The efficient solution set $E$ of {\rm (VLP)} is connected by line segments, i.e., for any $u, v$ in $E$, there eixst some elements $u_1,\dots, u_r$ of $E$, with $u_1=u$ and $u_r=v$, such that $[u_i, u_{i+1}] \subset E$ for $i=1,2,\dots,r-1$.
\end{theorem}
				
\begin{proof} According to Theorem \ref{scalar_theorem}, given any $u, v$ in $E$, one can find $\xi^*_0, \xi^*_1 \in \textrm{ri}K^*$  such that 
			\begin{equation*}
				u \in {\rm argmin} \left( {\rm (LP)}_{\xi^*_0}\right), \quad v \in {\rm argmin} \left( {\rm (LP)}_{\xi^*_1}\right).
			\end{equation*}
Since $\textrm{ri}K^*$ is a convex set, $\xi^*_t:=(1-t)\xi^*_0+t \, \xi^*_1$ belongs to  $\textrm{ri}K^*$ for every $t \in [0,1]$.  Noting that $\langle y^*, Mx \rangle = \langle M^*y^*,x \rangle$, by \cite[Proposition 3.6]{Luan_Yen_2015}, we can find finitely many nonempty generalized polyhedral convex sets $F_1, \dots, F_q$, which are subsets of $D$ such that, for any $y^* \in Y^*$ with ${\rm argmin} \left( {\rm (LP)}_{y^*}\right)$ is nonempty, the latter solution set coincides with one of the set $F_i$, $i=1,\dots,q$. 
				
By remembering the family $\{F_1,\dots,F_q\}$ we can assume that ${\rm argmin} \left( {\rm (LP)}_{\xi^*_0}\right)=F_1$. For each $i \in \{1,\dots,q\}$, put
\begin{equation*}
\Delta(i)=\left\{t \in [0,1] \, \mid \, F_i \subset {\rm argmin} \left( {\rm (LP)}_{\xi^*_t}\right)  \right\}.
\end{equation*}
To show that $\Delta(i)$ is a convex set, we take any $t_1, t_2 \in \Delta(i)$ and $\lambda \in (0,1)$. For $\bar{t}:=(1-\lambda)t_1+\lambda t_2$ and for any $u \in F_i$, one has
	\begin{equation*}
	\langle \xi^*_{\bar{t}}, Mx-Mu \rangle = (1-\lambda) \langle \xi^*_{t_1}, Mx- Mu \rangle + \lambda \langle \xi^*_{t_2}, Mx- Mu \rangle \geq 0, \quad \forall x \in D.
	\end{equation*}
Thus $u \in {\rm argmin} \left( {\rm (LP)}_{\xi^*_{\bar{t}}}\right)$. It follows that $F_i \subset {\rm argmin} \left( {\rm (LP)}_{\xi^*_{\bar{t}}}\right) $; so $\bar{t} \in \Delta(i)$. The convexity of $\Delta(i)$ has been proved.
				
If $\Delta(i)$ has only one element, it is closed. Now, suppose that $[t_1, t_2) \subset \Delta(i)$, $t_1 < t_2$. Since $\bar{t}:=(1-\lambda)t_1+\lambda t_2 \in \Delta(i)$ for all $\lambda \in (0,1)$, for any $u \in F_i$ and $x \in D$, one has
\begin{equation*}
0 \leq \langle \xi^*_{\bar{t}}, Mx-Mu \rangle = (1-\lambda) \langle \xi^*_{t_1}, Mx- Mu \rangle + \lambda \langle \xi^*_{t_2}, Mx- Mu \rangle. 
\end{equation*}
Letting $\lambda \to 1$, we obtain $\langle \xi^*_{t_2}, Mx- Mu \rangle \geq 0$ for all $u \in F_i$ and $x \in D$. This implies that $F_i \subset {\rm argmin} \left( {\rm (LP)}_{\xi^*_{t_2}}\right) $, i.e., $t_2 \in \Delta(i)$. Similarly, one can show that  if $(t_1, t_2] \subset \Delta(i)$ then $t_1 \in \Delta(i)$. We have thus proved that $\Delta(i)$ is a closed convex set for each $i=1,\dots,q$. Invoking Theorem 3.1 from \cite{Luan_Yen_2015}, it is easy to prove that the set of $y^* \in Y^*$ with ${\rm argmin} \left( {\rm (LP)}_{y^*}\right) \neq \emptyset$ is convex cone. Hence, for any $t \in [0,1]$, ${\rm argmin} \left( {\rm (LP)}_{\xi^*_{t}}\right)  \neq \emptyset$. It follows that $[0,1]=\bigcup\limits_{i=1}^q \Delta(i)$. Consequently, there exist some numbers $t_1, \dots, t_m$ from $[0,1]$ with $t_1 \leq \dots \leq t_{m}$, $t_1=0$, $t_{m}=1$, and $m-1$ indexes $i_1, \dots, i_{m-1}$ such that $[t_j,t_{j+1}] \subset \Delta(i_j)$ for all $j=1,\dots,m-1$. Clearly, $u \in F_{i_0}$ and there exists $i_r$ satisfying $v \in  F_{i_r}$.  Given  $u_{j} \in  F_{i_j}$ for $j=1,\dots,r$, where $u_1=u$ and $u_r=v$. For each $j=1,\dots,r-1$, since $t_{j+1}\in \Delta(i_j) \cap \Delta(i_{j+1})$,  it follows that $u_j \in {\rm argmin} \left( {\rm (LP)}_{\xi^*_{t_{j+1}}}\right) $ and $u_{j+1} \in {\rm argmin} \left( {\rm (LP)}_{\xi^*_{t_{j+1}}}\right) $. Hence, 
		\begin{equation*}
	[u_j,u_{j+1}] \subset  {\rm argmin} \left( {\rm (LP)}_{\xi^*_{t_{j+1}}}\right)  \subset E.
		\end{equation*}
We have already been proved that the line segments $[u_j,u_{j+1}], j=1,\dots,r-1$, connect the vectors $u, v$ in $E$.  The proof is complete.  
\end{proof}

\medskip
A similar relust for the weakly efficient solution set of {\rm (VLP)}.  
\begin{theorem}
	If ${\rm int}K \neq \emptyset$, then the weakly efficient solution set $E^w$ of {\rm (VLP)} is connected by line segments.
\end{theorem}

\begin{proof} 
  Let us first prove that the cone $K^*\setminus\{0\}$ is convex. Assume by contradiction that there exist $y^*_1$, $y^*_2 \in K^*\setminus\{0\}$ and $\lambda \in (0,1)$ satisfying $$y^*:=(1-\lambda)y^*_1+\lambda y^*_2 \notin K^*\setminus\{0\}.$$ Since $y^*_1$, $y^*_2 \in K^*$, which is a convex cone, $y^* \in K^*$; hence $y^*=0$. This implies that $\langle y^*_1, y \rangle =0$ for every $y \in K$. By ${\rm int}K \neq \emptyset$, it is not difficult to show that $\langle y^*_1, y \rangle =0$ for all $y \in Y$, which contradicts the assumption $y^*_1 \in K^*\setminus\{0\}$.
  
  Now, by \cite[Theorem 4.5]{Luan_Yen_2015}, we apply the proof scheme of Theorem \ref{eff_sol_arcwise}, with ${\rm ri}K^*$ being replaced by $K^*\setminus\{0\}$, to obtain $E^w$ is connected by line segments. 
\end{proof}

\section*{Acknowledgements}
\noindent 
The author would like to thank Professor Nguyen Dong Yen for his guidance and useful remarks.

\section*{Funding}
\noindent 
This research was supported by the Vietnam National Foundation for Science and Technology Development (NAFOSTED) under grant number 101.01-2014.37.


\begin{thebibliography}{99}

\bibitem{ABB_1953} Arrow KJ, Barankin EW, Blackwell D. Admissible points of convex sets. In: Contributions to the theory of games, vol.~2. Annals of Mathematics Studies, vol.~28:87--91. Princeton (NJ): Princeton University Press; 1953.

\bibitem{Jahn} Jahn J. Vector optimization. Theory, applications, and extensions. Berlin: Springer-Verlag; 2004.

\bibitem{Luc} Luc DT. Theory of vector optimization. Berlin: Springer-Verlag; 1989.

\bibitem{Yang_Yen_2010} Yang XQ, Yen ND. Structure and weak sharp minimum of the Pareto solution set for piecewise linear multiobjective optimization. J.~Optim.~Theory Appl. 2010;147:113--124.   

\bibitem{ZhengYang_2008} Zheng XY, Yang XQ. The structure of weak Pareto solution sets in piecewise linear multiobjective optimization in normed spaces. Sci. China Ser. A. 2008;151:1243--1256.

\bibitem{Zheng_1998} Zheng XY. Generalizations of a theorem of Arrow, Barankin, and Blackwell in topological vector spaces. J.~Optim.~Theory Appl. 1998;96:221--233.

\bibitem{Eich08a} Eichfelder G. Adaptive scalarization methods in multiobjective optimization. Berlin:  Springer-Verlag; 2008.

\bibitem{Hoa_Phuong_Yen_2005} Hoa TN, Phuong  TD, Yen ND. Linear fractional vector optimization problems with many components in the solution sets. J. Industr. Manag. Optim. 2005;1:477--486.

\bibitem{Huong_Yen_2014} Huong NTT, Yen ND. The Pascoletti--Serafini scalarization scheme and linear vector optimization. J.~Optim.~Theory Appl. 2014;162:559--576.

\bibitem{Jahn_1984} Jahn J. Scalarization in vector optimization. Math. Program. 1984;29:203--218.

\bibitem{Luc_1987} Luc DT. Scalarization of vector optimization problems. J.~Optim.~Theory Appl. 1987;55:85--102.

\bibitem{Luc2016} Luc DT. Multiobjective linear programming. An introduction. Cham (Switzerland): Springer; 2016.

\bibitem{Pascoletti_Serafini} Pascoletti A, Serafini P. Scalarizing vector optimization problems. J.~Optim.~Theory Appl. 1984;42:499--524.  

\bibitem{Yen_Phuong_2000} Yen ND, Phuong TD. Connectedness and stability of the solution sets in linear fractional vector optimization problems. In: Giannessi F. (eds) Vector variational inequalities and vector equilibria. Dordrecht (Holland): Kluwer Academic Publishers. 2000;479--489. 

\bibitem{Zheng_2000} Zheng XY. Scalarization of Henig proper efficient points in a normed space. J.~Optim.~Theory Appl. 2000;105:233--247.

\bibitem{Luan_Yen_2015} Luan NN, Yen ND. A representation of generalized convex polyhedra and applications. Preprint (Submitted); 2015. 

\bibitem{Bonnans_Shapiro_2000} Bonnans JF, Shapiro A. Perturbation analysis of optimization problems. New York (NY): Springer-Verlag; 2000. 

\bibitem{Rock_book_1970} Rockafellar RT. Convex analysis. Princeton (NJ): Princeton University Press; 1970.

\bibitem{Lang_1993} Lang S. Real and functional analysis. 3rd ed. New York (NY): Springer-Verlag; 1993.

\bibitem{Rudin_1991} Rudin W. Functional analysis. 2nd ed. New York (NY): McGraw Hill; 1991.

\bibitem{Podinovski_Nogin}  Podinovskji VV, Nogin  VD. Pareto optimal solutions in multicriteria optimization problems. Moscow: Nauka; 1982.			

\end{thebibliography}
\end{document}